\documentclass{amsart}

\usepackage{ifthen}
\usepackage{graphicx}
\usepackage{amssymb}
\usepackage{enumerate}
\usepackage{url}
\newtheorem{anyprop}{Anyprop}[section]

\newtheorem{theorem}[anyprop]{Theorem}
\newtheorem{lemma}[anyprop]{Lemma}

\newtheorem{corollary}[anyprop]{Corollary}

\theoremstyle{definition}

\newtheorem{definition}[anyprop]{Definition}

\newtheorem{example}[anyprop]{Example}

\newtheorem{remark}[anyprop]{Remark}

\theoremstyle{remark}

\numberwithin{equation}{section}

\usepackage{amscd}
\usepackage{amssymb}
\input xy
\xyoption{all}

\begin{document}
\title[ON LOCAL UNIFORM REDUCTIONS OF CERTAIN CLASSES OF VARIETIES]
{ON POSITIVE-CHARACTERISTIC SEMI-PARAMETRIC LOCAL REDUCTIONS OF FINITELLY GENERATED $\mathbb{Q}$-ALGEBRAS}

\author[Edisson Gallego]{Edisson Gallego}
\author[Danny A. J. G\'omez-Ram\'irez]{Danny Arlen de Jes\'us G\'omez-Ram\'irez}
\author[Juan D. Velez]{Juan D. V\'elez}
\address{University of Antioquia, Calle 67 \# 53-108, Medell\'in, Colombia}
\address{Vienna University of Technology, Institute of Discrete Mathematics and Geometry,
wiedner Hauptstaße 8-10, 1040, Vienna, Austria.}
\address{Universidad Nacional de Colombia, Escuela de Matem\'aticas,
Calle 59A No 63 - 20, N\'ucleo El Volador, Medell\'in, Colombia.}
\email{egalleg@gmail.com}
\email{daj.gomezramirez@gmail.com}
\email{jdvelez@unal.edu.co}



\begin{abstract}
We present a non-standard proof of the fact that the existence of a local 
(i.e. restricted to a point) characteristic-zero, semi-parametric lifting for a variety 
defined by the zero locus of polynomial equations over the integers is equivalent to 
the existence of a collection of local semi-parametric (positive-characteristic) reductions
of such variety for almost all primes (i.e. outside a finite set), and such that there exists a 
global complexity bounding all the corresponding structures involved. Results of this kind are a fundamental tool for 
transferring theorems in commutative algebra from a characteristic-zero setting to a positive-characteristic one.

\end{abstract}

\maketitle


\smallskip

\footnote{This paper should be cited as follows E. Gallego, D. A. J. G\'omez-Ram\'irez and and J. D. V\'elez. On Positive-Characteristic Semi-parametric Local Uniform Reductions of Varieties
over Finitely Generated $Q$-Algebras. Results in Mathematics, pp. 1-9, 2017. This article is distributed under the terms of the Creative Commons Attribution 4.0 International License (\url{http://creativecommons.org/licenses/ by/4.0/}), which permits unrestricted use, distribution, and reproduction in any medium, provided you give appropriate credit to the original author(s) and the source, provide a link to the Creative Commons license, and indicate if changes were made.

}
\section*{Introduction}

In this article we present a characterization of the fact that a finite system 
of polynomial equations over the integers has a local solution (i.e. a punctual one) over a characteristic 
zero $k$-algebra, where $k$ is a field, such that the first $n$-components of it represent a system of quasi-parameters 
(i.e. they generate a maximal ideal which induces a natural `residual' isomorphism with $k$). This 
equivalence is given in terms of the existence of positive-characteristic solutions with analogous properties and whose
complexity can be uniformly bounded. This result can be seen as a kind of local-global criterion 
for the existence of punctual solutions of elementary `Diophantine' varieties, which satisfies a sort 
of `semi-parametric' condition. A similar result was obtained by Hochster in \cite[Pag. 22]{hochster} with the restrictions 
that the algebras involved should be integral domains and the first $m$-components should generate a maximal ideal. 
In addition, Hochster's original proof uses quite intricate algebraic and homological (standard) methods.
As we will see in the last section, the former constraints can be essentially avoided due to the powerful non-standard methods that
we use in our proof, which are essentially based in the remarkable introduction of ultraproducts in commutative algebra due to the work of H. Schoutens  \cite[\S 1, 5]{bounds}, \cite[Ch. 4]{schoutens}.

\section{Preliminary facts}

We start by recalling the following definitions (see \cite{bounds}). 
Throughout this discussion, we will fix a monomial order in the
polynomial ring $k\left[ x_{1},\ldots ,x_{n}\right] $, where $k$ denotes a field. 

\begin{definition}
Let $R$ be a finitely generated $k$-algebra.

\begin{enumerate}
\item Let $I$ be an ideal of $k\left[ x_{1},\ldots ,x_{n}\right] .$ We will
say that $I$ has \textbf{complexity} \textbf{at most } $d$, if $n\leq d,$
and it is possible to choose generators\ for $I,$ $f_{1},\ldots ,f_{s},$
with $\deg f_{i}\leq d,$ for $i=1,\ldots ,s$.

\item We say $R$ has \textbf{complexity at most }$d$ if there is a
presentation of $R$ as a quotient $k\left[ x_{1},\ldots ,x_{n}\right] /I$,
with $I$ an ideal of complexity at most $d$.

\item If $J\subseteq R$ is an ideal, we will say that $J$ \textbf{\ has
complexity at most }$d$, if $R$ has complexity less than or equal to $d,$
and there exists a lifting of $J$ in $k\left[ x_{1},\ldots ,x_{n}\right] $,
let us say $J^{\prime },$ with complexity at most $d$.
\end{enumerate}
\end{definition}

\begin{remark}
In (1), the number of generators of $I$ may always be bounded in terms of $d$%
. In fact, without loss of generality we can assume that all the $f_{i}$ are
monic, and also that the leading terms of $f_{i}$ and $f_{j}$ are different
from each other, when $i\neq j$ (if they have same leading term, we can
change $f_{j}$ by $f_{j}-f_{i}$ and get a new set of generators for $I$
satisfying this last property). So, $s\leq D,$ where $D$ is the number of
monomials of degree $d$, $D=\left\vert \left\{ x_{1}^{r_{1}}\cdots
x_{n}^{r_{n}}\mid \Sigma _{i=1}^{n}r_{i}\leq d\right\} \right\vert .$ It is
then easy to see that $D=\tbinom{n+d}{n}\leq \tbinom{2d}{d}$, since $d\geq n$%
.
\end{remark}

Let $A=k\left[ x_{1},\ldots ,x_{n}\right] $ be the polynomial ring with a
fixed monomial order. For any polynomial $f\in A$ we will denote by $a_{f}$
the tuple of all the coefficients of $f$. When the complexity of $I$ is at
most $d,$ and $I=(f_{1},\ldots ,f_{s})$, by adding zeroes if necessary, we
may always assume that $s=D,$ where $D$ is the number defined above. Then,
the ideal $I$ can be encoded by a tuple of the form 
\begin{equation*}
a_{I}=\left( n,\underset{\text{of }f_{1}}{\underset{\text{D coefficients}}{%
\_\_\_\_\_\_\_\_}},\ldots ,\underset{\text{of }f_{s}}{\underset{\text{D
coefficients}}{\_\_\_\_\_\_\_\_}}\right) \in \mathbb{N}\times k^{D^{2}},
\end{equation*}%
where the monomials are listed according to the fixed monomial order.
Conversely, given one of those tuples, $a$, we can always reconstruct the
ideal it comes from. This ideal we shall denote by $\mathcal{I}\left(
a\right) $. Similarly, if $R$ is a $k$-algebra with complexity at most $d$,
then $R$ can be written as $k\left[ x_{1},\ldots ,x_{n}\right] /\mathcal{I}%
\left( a\right) .$ We will express this fact as $R=\mathcal{R}\left(
a\right) $.

For the sake of clarity for the reader 
and for introducing some important terminology used later, we will state explicitly  
some seminal results described in \cite{bounds}. Let us recall that if  
 $\Phi(x_1,\cdots,x_n)$ is a first-order formula, where $x_1,\cdots,x_n$ are the free 
variables, then the support of $\Phi$ with respect to a fixed interpretation $A$, 
denoted by $|\Phi|_{A}$, consists of all the $n-$tuples $(a_1,\cdots,a_n) \in A^n$ 
such that $\Phi(a_1,\cdots,a_n)$ is true in $A$. In addition, if 
$a_W=(n,a_1,\cdots,a_E)\in \mathbb{N}\times k^{D^{2}}$ is the code of some 
algebraic structure $W$, then we say that $a_W$ belongs to $|\Phi|_{k}$ if 
$(a_1,\cdots,a_E)\in |\Phi|_{k}$.

\begin{theorem}
(\cite{bounds}, Proposition 5.1.) For each $d,h>0$, there exists a formula $%
\mathrm{Height}_{d}=h$ such that for any field $k$, any finitely generated $%
k $-algebra $R$ of complexity at most $d,$ and any ideal $I\subseteq R$ of
complexity at most $d$, the height of $I$ is equal to $h$ if and only if $%
(a,b)\in |\mathrm{Height}_{d}=h|_{k}$, where $a,b$ are codes for $R$ and $I$%
, respectively.
\end{theorem}

\begin{example}
\label{altura de un ideal} Let $k$ be an algebraic closed field of
characteristic $0$, let $I\subseteq k[x_{1},\ldots ,x_{n}]$ be an ideal of
height $h$ and complexity at most $d.$ Let $a_{I}$ be a code for $I$ (so
that, $k\models \exists (\xi )($\textrm{Height}$_{d}(\xi )=h)$, namely 
\textrm{Height}$_{d}(a_{I})=h,$ holds in $k$). So, by Lefschetz's Principle $%
\mathbb{F}_{p}^{alg}\models \exists (\xi )($\textrm{Height}$_{d}(\xi )=h),$
for all $p>m$, where $m$ is some (fixed) integer. Let $a_{I^{\prime
}}^{\prime }$ be a tuple in $\mathbb{F}_{p}^{alg}$ for which, after
substitution, the sentence $\exists (\xi )($\textrm{Height}$_{d}(\xi )=h)$
holds true in $\mathbb{F}_{p}^{alg}$ for $p>m$. Then, by decoding $%
a_{I^{\prime }}^{\prime }$, we may find an ideal $I^{\prime }\subseteq \mathbb{%
F}_{p}^{alg}[x_{1},\ldots ,x_{n}]$ of height $h$ and with complexity at most 
$d$.
\end{example}

\begin{corollary}
\label{member} 
(\cite{schoutens}, Theorem 4.4.1.)Given $d>0$, there exists a formula \textrm{IdMem}$_{d}$ such
that for any field $k$, any ideal $I\subseteq k\left[ x_{1},\ldots ,x_{n}%
\right] $ and any $k$-algebra $R$, both of complexity at most $d$ over $k$,
it holds that $f\in IR$ if an only if $k\models $\textrm{IdMem}$%
_{d}(a_{f},a_{I})$. Here $a_{f}$ and $a_{I}$ denote codes for $f$ and $I$
respectively.
\end{corollary}

\begin{theorem}
(\cite{schoutens} Theorem 4.4.6) For any pair of integers $d,n>0$,
there exists a bound $b=b(d,n)$ such that for any field $k$, and any ideal $%
I\subseteq k[x_{1},\ldots ,x_{n}]$ of complexity at most $d$, its radical $J=%
\mathrm{Rad}(I)$ has complexity at most $b$. Moreover, $J^{b}\subseteq I,$ and 
$I$ has at most $b$ distinct minimal primes all of which are generated by
polynomials of degree at most $b$.
\end{theorem}

\begin{example}
\label{Formula-radical} Given $d,n>0,$ there exists a formula \textrm{Rad}$%
_{d}$ such that for any field $k$ and any pair of ideals $P,I\subseteq
k[x_{1},\ldots ,x_{n}]$ of complexity at most $d$, with $P$ a prime ideal
containing $I$, it holds: the radical of $I$ is $P$ (i.e., $\mathrm{Rad}(I)=P
$) if and only if $k\models \mathrm{Rad}_{d}(a_{I},a_{P})$. Here $a_{I},a_{P}
$ are codes for $I$ and $P$ respectively.\\*[0pt]
In fact, by the last theorem, we know that there exists a bound $b=b(d,n)$,
depending only on $d$ and $n$ such that $P^{b}\subseteq I.$ But this is
equivalent to saying that $\mathrm{Rad}(I)=P,$ since $\mathrm{Rad}(I)$ is
the intersection of all prime ideals containing $I$. It is then sufficient
for the formula $\mathrm{Rad}_{d}$ to express that the product of any set of 
$b$ elements between the bounded generators of $P$ lies in $I$. Now, this
can be done by means of a first order formula, by using Corollary \ref%
{member}.
\end{example}

\begin{remark}
\label{Inc} Using Corollary \ref{member}, it is easy to get for each $d,$
formulas \textrm{Inc}$_{d}$ and \textrm{Equal}$_{d}$ such that if $R$ is a
finitely generated $k$-algebra with complexity at most $d,$ and if $J$ and $I
$ are ideals of $R$ with complexity less than $d$, then $(a_{I,}a_{J})\in
\left\vert \text{Inc}_{d}\right\vert _{k}$ (resp. $(a_{I,}a_{J})\in
\left\vert \text{Equal}_{d}\right\vert _{k}$) if and only if $I\subseteq J$
(resp. $I=J$).
\end{remark}

\begin{theorem}
(\cite{schoutens} Theorem 4.4.4) For any pair of integers $d,n>0$,
there exists a bound $b=b(d,n)$ such that for any field $k$ and any ideal $%
P\subseteq k[x_{1},\ldots ,x_{n}]$ of complexity at most $d$, $P$ is a prime
ideal if and only if for any two polynomials $f,g$ of complexity at most $b$
which do not belong to $P$ then neither does their product.
\end{theorem}

\begin{remark}
Given $d,n>0$ there exists a formula \textrm{Prime}$_{d}$ such that for any
field $k$ and any ideal $P\subseteq k[x_{1},\ldots ,x_{n}]$ of complexity at
most $d$, $P$ is a prime ideal if and only if $k\models \mathrm{Prime}%
_{d}(a_{P})$. Where $a_{P}$ is a code for $P$.\\*[0pt]
The existence of this formula follows from the last theorem, and from
Corollary \ref{member}.
\end{remark}

\begin{example}
\label{Formula- idealprimo} Let $k$ be an algebraic closed field with 
${\rm char}(k)=0$, let $P$ be a prime ideal in $k[x_{1},\ldots ,x_{n}]$ of
complexity at most $d,$ and let $a_{P}$ be a code for $P$.

 So, $k\models \exists
(\xi )$\textrm{Prime}$_{d}(\xi )$, since \textrm{Prime}$_{d}(a_{P})$ holds
in $k$. By Lefschetz's Principle $\mathbb{F}_{p}^{alg}\models \exists (\xi )$%
\textrm{Prime}$_{d}(\xi )$ for all $p>m,$ for some $m$.

Furthermore, if $a_{P^{\prime }}^{\prime }$ is a tuple in $\mathbb{F}%
_{p}^{alg}$ for which the sentence $\exists (\xi )\mathrm{Prime}_{d}(\xi )$
is true in $\mathbb{F}_{p}^{alg},$ for a fix prime number $p>m$, then, by
decoding $a_{P^{\prime }}^{\prime },$ we may find a prime ideal $P^{\prime
}\subseteq \mathbb{F}_{p}^{alg}[x_{1},\ldots ,x_{n}]$ with complexity at most $%
d$.
\end{example}

\begin{example}
\label{Formula-Ideal Maximal} Given $d,n>0$ there exists a formula \textrm{%
MaxIdeal}$_{d,n}$ such that for any algebraic closed field $k$ and any ideal 
$m\subseteq k[x_{1},\ldots ,x_{n}]$ of complexity at most $d$ we have:\\*[0pt]
$m$ is a maximal ideal if and only if $k\models \mathrm{MaxIdeal}%
_{d,n}(a_{m})$, where $a_{m}$ is a code for $m$. In fact, by the
Nullstellensatz $m$ is maximal if and only if there exist $b_{1},\ldots
,b_{n}\in k$ such that $m=(x_{1}-b_{1},\ldots ,x_{n}-b_{n})$. Let us call $%
J_{\underline{b}}=(x_{1}-b_{1},\ldots ,x_{n}-b_{n})$. Then, the required
formula is: 
\begin{equation*}
\mathrm{MaxIdeal}(\xi ):(\exists b_{1},\ldots ,b_{n})(\mathrm{Equal}_{d}(\xi
,a_{J_{\underline{b}}})),
\end{equation*}%
where $\xi $ and $a_{J}$ must be replaced by the codes $a_{m}$ of $m,$ and $%
a_{J_{\underline{b}}}$ of $J_{\underline{b}},$ respectively.
\end{example}

\begin{lemma}
\label{Lema-construccion}

Let $\{F_{\alpha }(%
\underline{X},\underline{Y})\}_{\alpha =1,\ldots ,l}$ be a polynomial system
of equations with coefficients in $\mathbb{Z}$. Then, the following two conditions 
are equivalent:

(a) There exists a $k-$algebra $S=k[T_{1},\ldots ,T_{\nu }]/I$ (resp. integral domain) over an algebraic
closed field $k$ of characteristic $0$, where $I\subseteq k[T_{1},\ldots
,T_{\nu }]$ is an ideal (resp. prime ideal) of complexity at most $d$; and $m\subseteq
k[T_{1},\ldots ,T_{\nu }]$ a prime ideal (resp. maximal) with complexity at most $d$
and height $n$ in $S$. And, there exists a tuple $(\underline{x},%
\underline{y})=(x_{1},\ldots ,x_{n},y_{1,}\ldots ,y_{r})$ of
elements in $S$ such that:

\begin{enumerate}
\item $\mathrm{Rad}(\underline{x})=m$ in $S$.

\item $F_{\alpha}(\underline{x},\underline{y})=0$, for all $\alpha=1,\ldots,
l $.
\end{enumerate}

(b) There exists $c, d\in \mathbb{N}$, such that for any prime number $p\geq c$, 
we can always construct a $\mathbb{F}_{p}^{alg}$-algebra (resp. $\mathbb{F}_{p}^{alg}$-domain) $S^{\prime }=\mathbb{F}%
_{p}^{alg}[T_{1},\ldots ,T_{\nu }]/I^{\prime },$ with $I^{\prime }\subseteq 
\mathbb{F}_{p}^{alg}[T_{1},\ldots ,T_{\nu }]$ an ideal (resp. prime ideal) of complexity at
most $d$, such that:

\item There exists $m^{\prime }\subseteq \mathbb{F}_{p}^{alg}[T_{1},\ldots
,T_{\nu }],$ a prime ideal (resp. maximal) with complexity at most $d$ and height $n$ in $%
S^{\prime }$, and $(\underline{x^{\prime
}},\underline{y^{\prime }})=(x_{1}^{\prime },\ldots ,x_{n}^{\prime
},y_{1,}^{\prime }\ldots ,y_{r}^{\prime }),$ a tuple of elements in $%
S^{\prime }$ such that:

\begin{enumerate}
\item $\mathrm{Rad}(\underline{x^{\prime }})=m^{\prime }$ in $S^{\prime }$.

\item $F_{\alpha }(\underline{x^{\prime }},\underline{y^{\prime }})=0$, for
all $\alpha =1,\ldots ,l$.
\end{enumerate}
\end{lemma}

\begin{proof}
$(\Rightarrow)$ First, let us consider the cases where $I$ is a prime ideal and $m$ is a prime ideal,
 i.e, $I\subseteq m,$ and $\mathrm{ht}(m)=\omega =n+%
\mathrm{ht}(I)$ in $k[T_{1},\ldots ,T_{\nu }]$, where $\omega$ is an integer $\leq \nu$
 (resp. $\omega = \nu$, if (and only if) $m$ is maximal).
The hypothesis above may be expressed by means of a first order formula $%
\Phi _{d}$ such that when we evaluate it on the codes of $S,m,I,\{F_{\alpha
}(\underline{X},\underline{Y})\}$, respectively; $k\models \Phi _{d}$ if and
only if $m$ is a maximal ideal of height $n$ in $S$, $I$ is a prime ideal
contained in $m$, and (1), (2) are satisfied. This formula may be
explicitly given as:

\begin{equation*}
\Phi _{d}:(\exists a_{m},a_{I},a_{(\underline{x},\underline{y})})(\mathrm{%
Prime}_{d}(a_{I})\wedge \mathrm{Prime}_d(a_m) 
\end{equation*}%
\begin{equation*}
\mathrm{Inc}_{d}(a_{I},a_{m}) \newline
\wedge \mathrm{Height}_{d}(a_{m})=\omega \wedge \mathrm{Height}_{d}(a_{I})=\omega
-n\wedge \mathrm{Rad}_{d}(a_{(\underline{x})},a_{m})\wedge 
\end{equation*}%
\begin{equation*}
\mathrm{IdMem}_{d}(F_{\alpha }(\underline{x},\underline{y}),a_{I})),
\end{equation*}%

where the degrees of parameters $(\underline{x},\underline{y})$ in $\Phi _{d}
$ is bounded by the degrees of fixed liftings in $k[T_{1},\cdots ,T_{\nu}]$
of the actual existing parameters $(\underline{x},\underline{y})$ in $S$.

Now, by Lefschetz's principle, $k\models \Phi _{d}$, if and only if $\mathbb{%
F}_{p}^{alg}\models \Phi _{d}$, for any prime $p$ large enough. As we
discussed in Examples \ref{Formula-radical}, %
\ref{Formula- idealprimo}, \ref{altura de un ideal} and Remark \ref{Inc};
there are $\mathbb{F}_{p}^{alg}-$tuples $a_{m}^{\prime },a_{I}^{\prime }$
and $a_{(\underline{x},\underline{y})}^{\prime }$ which codify a prime
ideal $m^{\prime }$ (resp. $\mathrm{Height}_{d}(a_{m})=\nu$ codifies additionally the maximality of $m$), a prime ideal $I^{\prime }$, and a system of elements $%
\underline{x},\underline{y}$, satisfying all the required conditions in $%
\mathbb{F}_{p}^{alg}[T_{1},\ldots T_{\nu }]/I^{\prime }$.

So, $S^{\prime }=\mathbb{F}^{alg}_{p}[T_{1},\ldots T_{\nu}]/I^{\prime }$ is
the required $\mathbb{F}^{alg}_{p}$-algebra. Finally, it is clear from above
that $S^{\prime }$ might be constructed of characteristic equal to $p$, for
any prime $p$ big enough.

For the cases where $I$ is not necessarily a prime ideal  
we just eliminate the formula $Prime_d(-)$ on $\Phi(d)$ from the former proof, accordingly.

$(\Leftarrow)$ If there exists a global complexity $d$ satisfying (b) for all prime numbers $p\geq c$, then 
we can construct in each of the former cases a suitable formula $\Phi_d$, such that $\mathbb{F}^{alg}_{p} \models \phi_d$ for all $p\geq r$. 
Thus, by Lefschetz's principle for any algebraic closed field $k$ of characteristic zero, we get $k \models \Phi_d$. 
So, as we have seen before in the examples we can construct the corresponding $k$-structures described in condition (a),
which finishes the proof.

\end{proof}

\begin{remark}
Note that in the former lemma the condition (a) can be re-phrased in a more general way requiring that for a fixed complexity $d$ and any algebraic closed field $k$ 
characteristic zero, those $k$-structures mention there exists. In this case, the proof would be essentially the same due to the usage of the Lefschetz's principle.
\end{remark}

\section{Main Result}

The main theorem of this paper is the following: 

\begin{theorem}
\label{Tma hochster}
Let $\{F_{\alpha }(\underline{X},\underline{Y})\}$%
, with $\alpha =1,\ldots ,l$, be a polynomial system of equations with
coefficients in $\mathbb{Z}$. Then, the following two conditions are equivalent:

(a) There exists a field of characteristic zero $k$, a finitely generated $k$-algebra (resp. domain) $S$;
 a prime ideal $m\subseteq S$ (resp. maximal)
 with height $n$ and $(\underline{x},\underline{y}%
)=(x_{1},\ldots ,x_{n},y_{1,}\ldots ,y_{r})$ a tuple of elements in $S$
such that:

\begin{enumerate}
\item $\mathrm{Rad}(\underline{x})=m.$

\item $F_{\alpha }(\underline{x},\underline{y})=0,$ for every $\alpha .$

\item (If $m$ is a maximal ideal) $k\subseteq S\overset{\pi }{\rightarrow }{S/\!m}$ is an isomorphism.
\end{enumerate}

(b)There exists a global complexity $d$ such that for all prime numbers $p$
not belonging to a finite set,
there exists a field $L$ of prime characteristic, a finitely generated 
$L$-algebra (resp. domain) $S^{\prime }$, a prime ideal $m^{\prime }$ of $S^{\prime }$ (resp. maximal)
with height $n,$ and elements $(\underline{x}^{\prime },\underline{y}%
^{\prime })=(x_{1}^{\prime },\ldots ,x_{n}^{\prime },$ $y_{1,}^{\prime
}\ldots ,y_{r}^{\prime })$ in $S^{\prime }$, all with complexity at most $d$
 such that:

\begin{enumerate}
\item $\mathrm{Rad}(\underline{x}^{\prime })=m^{\prime }.$

\item $F_{\alpha }(\underline{x}^{\prime },\underline{y}^{\prime })=0$, for
all $\alpha .$

\item (If $m^{\prime}$ is a maximal ideal) $%
\begin{array}{ccccc}
L & \overset{i}{\hookrightarrow } & S^{\prime } & \overset{\pi }{%
\longrightarrow } & S^{\prime }/m^{\prime }%
\end{array}%
,$ with $\pi \circ i$ an isomorphism.
\end{enumerate}

\end{theorem}

\begin{proof}

\textbf{First step}: Reduction of the condition (a) to the case where $k$ is algebraically closed, 
when $m$ is a maximal ideal. The remaining cases can be proved in a basically the same way.

Let $R$ be a finitely generated $k$-algebra, where $k$ is any field
of characteristic zero. Let $m\subseteq R$ a maximal ideal of height $n$, and
let $x_{1},\ldots ,x_{n}$ be elements in $R$ such that $\mathrm{Rad}%
(x_{1},\ldots ,x_{n})=m$, and such that conditions (1)-(3) hold in $R$. Let $%
\overline{k}$ be an algebraic closure of $k$, and define $R^{\prime }=%
\overline{k}\otimes _{k}R$. We notice that $R\longrightarrow R^{\prime }$ is
a faithfully flat extension, and therefore $R$ injects into $R^{\prime }$.
Moreover, $mR^{\prime }\cap R=m$ (\cite{Matsumura} Theorem 7.5, (2), page
49). Consequently, there is a prime ideal $q\subseteq R^{\prime }$ such that $%
q\cap R=m$. We notice that $q$ has to be maximal: Since $R$ is a finitely
generated $k$-algebra domain, all its maximal ideals have the same height,
equal to the Krull dimension of $R$. Hence, by Noether Normalization
Theorem, there are algebraically independent elements $a_{1},\ldots ,a_{n}$
in $R$ such that $A=k[a_{1},\ldots ,a_{n}]\subseteq R$ is a module-finite
extension. But this implies that $\overline{k}[a_{1},\cdots ,a_{n}]\subseteq 
\overline{k}\otimes _{k}R$ is also module-finite extension, and consequently 
$\dim (\overline{k}\otimes _{k}R)=\dim R$. From this, we see that $q$ has to
be a maximal ideal in $R^{\prime },$ since $R^{\prime }/q$ is an integral
domain of dimension zero, i.e., a field.

Now, in $R_{q}^{\prime },$ the ideal $mR_{q}^{\prime }$ is $qR_{q}^{\prime }$%
-primary. Thus, there is a power $n>0$ such that $q^{n}R_{q}^{\prime
}\subseteq mR_{q}^{\prime }$. On the other hand, there is a power $l>0$, such
that $m^{l}R\subseteq (x_{1},\ldots ,x_{n})R.$ Thus, $q^{n+l}R_{q}^{\prime
}\subseteq (x_{1},\ldots ,x_{n})R_{q}^{\prime }$. After inverting a finite
number of elements in $R^{\prime }$ we may assume that by localizing at a
single element $u\in R^{\prime }-q,$ the inclusion $q^{n+l}R_{u}^{\prime
}\subseteq (x_{1},\ldots ,x_{n})R_{u}^{\prime }$ still holds. We let $%
R^{\prime \prime }$ be the localized ring $R_{u}^{\prime }$.

This ring is a finitely generated $\overline{k}$-algebra extension of $R$ of
the same dimension. Moreover, the ideal $m^{\prime \prime }=qR^{\prime
\prime }$ is maximal, and

\begin{equation*}
q^{n+l}R^{\prime \prime }\subseteq (x_{1},\ldots ,x_{n})R^{\prime \prime }.
\end{equation*}%
Therefore, $\mathrm{Rad}(x_{1},\ldots ,x_{n})R^{\prime \prime }=m^{\prime
\prime }R^{\prime \prime }.$ Let $Q\subseteq R^{\prime \prime }$ be a minimal
prime ideal of $R^{\prime \prime }$ included in $m^{\prime \prime },$ and
such that $\dim (R^{\prime \prime }/Q)=\dim R^{\prime \prime }$. Thus, if we
let $S$ be the ring $R^{\prime \prime }/Q,$ and let $\eta =m^{\prime \prime
}S$, then, $S$ is a f.g. $\overline{k}$-algebra domain, with $ht(\eta )=n$,
and $\mathrm{Rad}(x_{1},\ldots ,x_{n})S=\eta S.$ Thus, condition 1 holds in $%
S$.

Besides, since there is a ring homomorphism $R\longrightarrow S,$ it is then
clear that condition 2 also holds in $S$. Finally, the Nullstellensatz
implies that condition 3 is true in $S$. So, we may replace $R$ by $S$.

\textbf{Second step}:

$(\Rightarrow)$ Let us take a presentation for $S$, say $S=k[T_{1},\ldots ,T_{\nu }]/I$.
Since $S$ is an algebra (resp. integer domain), we have that $I\subseteq k[T_{1},\ldots ,T_{\nu
}]$ is a ideal (resp. prime ideal). By the hypothesis, there exists a prime ideal $%
m\subseteq k[T_{1},\ldots ,T_{\nu }]$  with height $n$ in $S$ (resp. maximal, i.e., $%
\mathrm{ht}(m)=\nu =n+\mathrm{ht}(I)$ in $k[T_{1},\ldots ,T_{\nu }]$), and
there exists a tuple of elements in $S$, $(\underline{x},\underline{y}%
)=(x_{1}(\underline{t}),\ldots ,x_{n}(\underline{t}),y_{1,}\ldots ,y_{r}(%
\underline{t}))$ such that $\mathrm{Rad}(\underline{x})=m$ and 
$F_{\alpha}(\underline{x},\underline{y})=0$, for all $\alpha=1,\ldots,
l.$ Let us note that since we can suppose by the first step that $k$ is algebraically closed, 
the condition requiring that $k\subseteq S\overset{\pi }{\rightarrow }{%
S/\!m}$ is an isomorphism, turns out to be (trivially) satisfied.

Let $d>0$ be an integer that bounds all the complexities of the objects
mentioned above. So, the proof follows from Lemma \ref{Lema-construccion}.

$(\Leftarrow)$ Assume that there exists a uniform complexity $d$, and $c\in \mathbb{N}$ such that for 
any characteristic $p\geq c$ the $L$-structures of condition (b) exist. Then, by 
Lemma \ref{Lema-construccion} there exists an (algebraically closed) field of 
characteristic zero $k$, a finitely generated $k$-algebra (resp. domain) $S$;
 a prime ideal $m\subseteq S$ (resp. maximal)
 with height $n$, and $(\underline{x},\underline{y}%
)=(x_{1},\ldots ,x_{n},y_{1,}\ldots ,y_{r})$ a tuple of elements in $S$
such that $\mathrm{Rad}(\underline{x})=m$, and $F_{\alpha }(\underline{x},\underline{y})=0,$ 
for every $\alpha$. Finally, if the $m^{\prime}$ are maximal ideals, then by the Nullstellensatz,
  $k\subseteq S\overset{\pi }{\rightarrow }{S/\!m}$ is an isomorphism. 
\end{proof}
\begin{remark}
As pointed out in the introduction, the importance of the former Theorem lies in the fact that 
it generalizes Hochster's theorem, as it first appeared in \cite[Pag. 22]{hochster}. Now, Hochster's 
remarkable result provides a natural bridge for proving statements about equicharacteristic rings in characteristic
zero, by first reducing to the case of prime characteristic. In addition, Hochster's
theorem combined with M. Artin's approximation theorem \cite{artin} turns out to
be especially powerful. The full existence of Big Cohen Macaulay modules,
for instance, depends directly on Hochster's reduction method. For a survey
of different applications of this technique the reader may consult \cite[Ch. 4, 5, 7-9]{huneke1}. 

On the other hand, Hochster's theorem becomes especially relevant for the
construction of a theory of Tight Closure in characteristic zero
\cite[Ch. 3]{hochster1999}, \cite[Pag. 94]{huneke1} (for further readings 
see also \cite{huneke2} and \cite{epstein}).

\end{remark}

\section*{acknowledgements}
We would like to thank the reviewer and the editor for useful comments on the initial version of this manuscript. In addition, Danny Arlen de Jes\'us G\'omez-Ram\'irez would like to thank C. Thompson for valuable suggestions during the writing process of this article. Finally, D. A. J. G\'omez-Ram\'irez was supported by the Vienna Science and Technology Fund (WWTF) as part of the Vienna Research Group 12-004.



\begin{thebibliography}{9}
\bibitem{artin}
M. Artin, \emph{Algebraic approximation of structures over complete
  local rings}, Publications math{\'e}matiques de l'IH{\'E}S \textbf{36},
  23--58(1969)
\bibitem{eisenbud} D. Eisenbud, \emph{Commutative Algebra with a View Toward
Algebraic Geometry}, vol. 150, GTM, Springer Science \& Business Media, New York (1995)

\bibitem{epstein}
N. Epstein, \emph{A guide to closure operations in commutative algebra},
  Progress in commutative algebra \textbf{2}, 1--37 (2012)

\bibitem{hochster} 
M. Hochster, \emph{Topics in the Homological Theory of Modules over Commutative Rings},
Proceedings of the Nebrasca C.B.M.S. Conference, (Lincoln, Nebraska, 1974), Amer. Math. Soc., Providence, 75 pp, (1975)

\bibitem{hochster1999}
M. Hochster and C. Huneke, \emph{Tight closure in equal characteristic
  zero}, preprint (1999).

\bibitem{huneke1}
C. Huneke, \emph{Tight closure and its applications}, no.~88, American
  Mathematical Soc. (1996)

\bibitem{huneke2}
C. Huneke, \emph{Tight closure, parameter ideals, and geometry}, Six lectures on
  commutative algebra, Springer Science \& Business Media, pp.~187--239 (1998)

\bibitem{Matsumura} 
H. Matsumura, Commutative Algebra, Mathematics
Lecture Note Series, 56, Benjamin/Cummings Publishing Company, San Francisco (1980)

\bibitem{bounds}
H. Schoutens, Bounds in Cohomology, Israel J. Math. 116, 125-169 (2000)

\bibitem{schoutens} 
H. Schoutens, The Use of Ultraproducts in
Commutative Algebra, Lecture Notes in Mathematics, Springer Science \& Business Media, New York
(2010)

\end{thebibliography}
\end{document}